\documentclass{article}

\usepackage{amsmath}
\usepackage{amsfonts}
\usepackage{amsthm}

\newtheorem{lemma}{Lemma}
\newtheorem{prop}{Proposition}
\newtheorem{theorem}{Theorem}

\title{Large-time uniqueness in a data assimilation problem for Burgers' equation}
\author{Graham Cox
{\small ghcox@email.unc.edu}}

\begin{document}
\maketitle

\begin{abstract}
There is currently a great deal of interest in the $4D$-Var data assimilation scheme, in which one uses observational data to find the optimal initial condition for a differential equation by minimizing a cost function over the set of all possible initial states. For nonlinear models this cost function can be nonconvex, and so the uniqueness of minimizers is not guaranteed. In this paper we apply $4D$-Var to Burgers' equation and prove that, once a sufficient amount of data has been collected, there can be at most one physically reasonable minimizer to the variational problem.

\end{abstract}

\section{Introduction}

We consider here the problem of finding the ``best" initial condition $u$ for Burgers' equation with Dirichlet boundary conditions
\begin{align}
	y_t + y y_x &= \epsilon y_{xx} 	\label{eqn:Burgers} \\
	y(0,t) &= y(1,t) = 0 \nonumber \\
	y(x,0) &= u(x) \nonumber
\end{align}
in the presence of noisy observational data. It is known that a unique solution to the forward problem exists for any initial state $u$ in the space $V := H^1_0(0,1)$, with norm defined by $\|u\|^2_V = \int_0^1 u_x^2 dx$. We denote this solution by $y(u)$. Since Burgers' equation can be transformed to the heat equation via the Cole--Hopf transformation, we have that $y(u) \in C^{\infty}([0,1] \times (0,\infty))$. The observational data are assumed to be given by a bounded linear observation operator $H: V \rightarrow Z$, into a Hilbert space $Z$.

Given continuous (in time) observational data $z(t)$ and a fixed maximal observation time $T$, we define a cost functional
\begin{align}
	J_T(u) = \int_0^T \| H y(u) - z \|_Z^2 dt + \beta \|u - \bar{u}\|_V^2,
\end{align}
where $\bar{u}$ is a fixed background term and $\beta$ is a positive constant. Our goal is then to understand the existence of minimizers for $J_T$---functions $u_0 \in V$ satisfying
\begin{align}
	J_T(u_0) = \inf \{ J_T(u) : u \in V \}.
\label{eqn:min}
\end{align}
White \cite{W93} has established the following fundamental results.

\begin{theorem}[White '93] There exists a solution of (\ref{eqn:min}) for every $T \geq 0$. Moreover, there is a positive constant $T_1 = T_1(\epsilon, \beta, z, \|H\|)$ such that the solution is unique provided $T \leq T_1$.
\label{thm:White}
\end{theorem}

White's proof shows $J_T$ in fact has a unique critical point for $T$ sufficiently small, which is a stronger result than is stated above as it rules out the possibility of multiple local minima. The result is not expected to hold for all times $T > 0$ due to the nonlinearity present in (\ref{eqn:Burgers}). However, we have been able to show that, for sufficiently large observation times, one in fact regains uniqueness of the minimizer.

\begin{theorem} Suppose the observational data satisfy $\| z(t) \|_Z \leq c$ for all $t \geq 0$. Then for each $K > 0$ there exists a constant $T_2 = T_2 (\epsilon, \beta, c, K, \|H\|)$ such that, for any $T \geq T_2$, the functional $J_T$ has at most one critical point satisfying $\|u\|_{L^2} \leq K$.
\label{thm:unique}
\end{theorem}

The result says that a bounded set in $L^2(0,1)$ can contain at most one solution of (\ref{eqn:min}) once $T$ is large enough. This conclusion is satisfying as far as applications are concerned, since $u$ typically represents a quantity subject to certain physical constraints. For instance, a predicted initial state in which the fluid velocity $u$ is, on average, greater than the speed of light, would be of little practical use.

One can easily obtain an \textit{a priori} bound on minimizers as $J_T(u_0) \leq J_T(0)$, whence $\|u_0 - \bar{u}\|_V^2 \leq \beta^{-1} Z(T)^2 + \| \bar{u}\|_V^2$, where we have defined the integrated data function $Z(T)^2 = \int_0^T \|z(t)\|_Z^2 dt$. However, this estimate only appears to be useful for small $T$ unless one makes rather strong assumptions concerning the asymptotic behavior of $Z(T)$. This explains the $L^2$-bound on $u$ required in Theorem \ref{thm:unique}, which is not necessary for Theorem \ref{thm:White} to hold. Long-time uniqueness can be obtained in the absence of the bound $\|u\| \leq K$ if one assumes $Z(T) = o(T)$. Physically this means that the measurement error decreases over time, \textit{e.g.} $\| z(t) \|_Z = O(t^{-\delta})$ for any $\delta > 0$. Since this scenario seems unlikely to arise in practice, we do not dwell on it any longer.

We finally observe that Theorem \ref{thm:unique} holds more generally for polynomially bounded data, $Z(T) = \mathcal{O} \left( T^N \right)$, with the required observation time $T_2$ then additionally depending on $N$. However, the most physically relevant case is that of uniformly bounded observations with $\| z(t) \|_Z \leq c$, hence $Z(T) \leq c \sqrt{T}$.

\textbf{Acknowledgments:} The author would like to thank Damon McDougall for numerous enlightening conversations throughout the preparation of this work. This research has been supported by the Office Naval Research under the MURI grant N00014-11-1-0087.

\section{Fixed-point interpretation of the Euler--Lagrange equation}
In this section we review the Euler--Lagrange equation for $J_T$, and show that it can be interpreted as a fixed-point equation for a nonlinear compact operator $S_T: V \rightarrow V$. From \cite{W93} we know that the derivative of $J_T$ evaluated at $u$, in the direction $v \in V$, satisfies
\begin{align}
	\frac{1}{2} DJ_T(u)(v) &= \left< p(0), v \right> + \beta \left< (u - \bar{u})_x, v \right> \nonumber \\
		&= \left< p(0) - \beta (u - \bar{u})_{xx}, v \right>,
		\label{eqn:gradJ}
\end{align}
where $p(0)$ denotes the solution to the adjoint equation
\begin{align}
	-p_t - p_x y - \epsilon p_{xx} =& - \left[ H^* (Hy - z) \right]_{xx} \label{eqn:adjoint} \\
	p(0,t) &= p(1,t) = 0 \nonumber \\
	p(T) &= 0 \nonumber
\end{align}
evaluated at time $t = 0$. Thus $u$ is a critical point of $J_T$ precisely when $p(0) = \beta (u - \bar{u})_{xx}$. Note that this is a nonlinear equation for $u$, because $p$ depends linearly on $y$, which in turn depends nonlinearly on $u$. Writing these dependencies more explicitly as $p = p[y]$ and $y = y(u)$, we have that $u$ is a critical point if and only if is is a fixed point of the nonlinear map
\begin{align}
	S_T(u) = \bar{u} + \left. \frac{1}{\beta} \Delta^{-1}  p \left[ y (u) \right] \right|_{t=0}
\end{align}
acting on $V$.

Our goal is thus to show that $S_T$ is a contraction mapping under the conditions of Theorem \ref{thm:unique}. This is accomplished in the following section by establishing certain decay estimates for solutions of Burgers' equation. This large-time decay is strong enough to ameliorate the growth in the corresponding solutions of the adjoint equation, which is caused by the data-dependent forcing term on the right-hand side of equation (\ref{eqn:adjoint}), and so we are able to obtain the desired bound on $S_T$ when $T$ is large.


\section{The contractive estimate}
To study the contractivity of $S_T$, we fix $u_1$ and $u_2$ in $V$ and define $v = u_1 - u_2$, $\delta = y_1 - y_2$ and $\rho = p_1 - p_2$, where $y_i := y(u_i)$ and $p_i = p[y_i]$ for $i = 1,2$. For the remainder of the paper we let $\| \cdot \|$ denote the $L^2$ norm on $[0,1]$, hence $\|v\|_V = \| v_x \|$ for any $v \in V$. We thus wish to estimate
\begin{align*}
	\| S_T(u_2) - S_T(u_1) \|_V = \frac{1}{\beta}  \| \Delta^{-1} \rho(0) \|_V
\end{align*}
above in terms of $\| u_2 - u_1 \|_V = \| v_x \|$. It is a consequence of standard elliptic estimates that there is a positive constant $K$ such that
\begin{align*}
	\| \Delta^{-1} \rho \|_V \leq K \| \rho \|
\end{align*}
holds uniformly for all $\rho \in V$, so by the Poincar\'{e} inequality it would suffice to establish the estimate
\begin{align}
	\| \rho(0) \| \leq C(T) \| v \|
\end{align}
with some positive constant satisfying $C(T) \rightarrow 0$ as $T \rightarrow \infty$.

A key ingredient in the following estimates is the fact that any solution $y = y(u)$ of Burgers' equation is contained in $L^2(0,\infty; V)$. More explicitly, we have from \cite{W93} that
\begin{align}
	\int_0^T \| y_x \|^2 dt \leq \frac{ \| u \|^2}{2 \epsilon}
\label{eqn:ybound}
\end{align}
for any $T \geq 0$.

This will frequently be used in conjunction with the following version of Gronwall's inequality.

\begin{lemma} Suppose $u'(t) \leq [a(t) - a_0] u(t) + b(t)$, where $a(t)$, $b(t)$ and $a_0$ are nonnegative. If $A := \int_0^{\infty} a(t) dt < \infty$, then
\begin{align*}
	u(T) \leq e^{A - a_0 T} \left[ u(0) + \int_0^T b(t) a^{a_0 t} dt \right]
\end{align*}
for all $T \geq 0$.
\label{lemma:gronwall}
\end{lemma}

There are no sign restrictions of $u(t)$, though it will always be nonnegative in our applications. If $u(0) = 0$ the lemma immediately implies $u(T) \leq e^A \int_0^T b(t) dt$. A more subtle application will be given in the proof of Proposition \ref{prop:rho}.

\begin{proof} We define the integrating factor
\begin{align*}
	\mu(t) = \exp \left\{ \int_0^t [a(s) - a_0] ds \right\}
\end{align*}
and observe that $e^{-a_0 t} \leq \mu(t) \leq e^{A-a_0 t}$ for all $t \geq 0$. It follows that
\begin{align*}
	\frac{d}{dt} \left( \frac{u(t)}{\mu(t)} \right) \leq b(t) e^{a_0 t}
\end{align*}
and the proof is completed by integrating both sides of the above inequality.
\end{proof}

We now investigate boundedness of solutions to the adjoint equation (\ref{eqn:adjoint}), with Dirichlet boundary conditions and terminal condition $p(T) = 0$.

\begin{prop} There exists a constant $C$, depending only on $\epsilon$, $\| u \|$ and $\| H \|$, such that
\begin{align*}
	\| p(t) \| \leq C \sqrt{1 + Z(T)^2- Z(t)^2}
\end{align*}
for any $T \geq 0$ and all $0 \leq t \leq T$.
\label{prop:Pbound}
\end{prop}

\begin{proof}
We differentiate $\| p(t) \|^2$, finding that
\begin{align*}
	-\frac{1}{2} \frac{d}{dt} \|p\|^2 &= - \epsilon \int_0^1 p_x^2 \ dx + \int_0^1 p p_x y \ dx + \left< Hp, Hy - z \right>_Z \\
		&\leq \left( \frac{\lambda_1 \|H\|}{2} + \frac{\lambda_2 \| y \|_{\infty}}{2} - \epsilon \right) \|p_x\|^2
		+ \frac{\| y \|_{\infty}}{2 \lambda_2} \|p \|^2 + \frac{ \|H\|}{2 \lambda_1} \|Hy - z\|_Z^2
\end{align*}
for any positive $\lambda_1$ and $\lambda_2$. We then choose $\lambda_1 = 2 \epsilon / 3 \| H \|$ and $\lambda_2 = 2 \epsilon / 3 \|y \|_{\infty}$, so that
\begin{align*}
	-\frac{d}{dt} \|p\|^2 &\leq 
	\left( \frac{3 \| y \|_{\infty}^2}{2 \epsilon} - \frac{2 \epsilon^2 \pi}{3} \right) \|p \|^2 + \frac{ 3 \|H\|^2}{2 \epsilon} \|Hy - z\|_Z^2.
\end{align*}
We are now in a position to apply Lemma \ref{lemma:gronwall} to the function $\tilde{p}(t) := p(T-t)$, because
\begin{align*}
	\int_0^{\infty} \| y(t) \|_{\infty} ^2 dt \leq \frac{\| u \|^2}{2 \epsilon}
\end{align*}
by Equation (\ref{eqn:ybound}). It follows that
\begin{align*}
	\|p(t)\|^2 \leq C \int_t^T \| Hy - z \|_Z^2 dt,
\end{align*}
where $C$ depends on $\epsilon$, $\| u \|$ and $\| H \|$. We finally estimate $\| Hy - z\|^2 \leq 2 \| H \|^2 \| y_x \|^2 + 2 \|z\|_Z^2$ and integrate, again applying Equation (\ref{eqn:ybound}) to conclude that the first term is bounded.
\end{proof}

We next estimate the difference $\delta = y_1 - y_2$ between two solutions of Burgers' equation with initial values $u_1$ and $u_2$, respectively. 
\begin{prop} There exist a constant $A$, depending only on $\epsilon$, $\| u_1 \|$ and $\| u_2 \|$, such that
\begin{align*}
	\| \delta(t) \|^2 \leq A \| u_1 - u_2 \|^2 e^{-\epsilon \pi^2 t}
\end{align*}
for all $t \geq 0$.
\label{prop:deltabound}
\end{prop}

\begin{proof}
We first compute
\begin{align*}
	\delta_t &= y_{1t} - y_{2t} \\
		&= \epsilon y_{1xx} - y_1 y_{1x} - (\epsilon y_{2xx} - y_2 y_{2x} ) \\
		&= \epsilon (y_1 - y_2)_{xx} + y_2 y_{2x} - y_1 y_{2x} + y_1 y_{2x} - y_1 y_{1x} \\
		&= \epsilon \delta_{xx} - \delta y_{2x} - y_1 \delta_x
\end{align*}
and hence find
\begin{align*}
	\frac{1}{2} \frac{d}{dt} \| \delta \|^2 & \leq - \epsilon \| \delta_x \|^2 + 2 \| y_2 \|_{\infty} \| \delta \| \| \delta_x \| + \| y_1 \|_{\infty} \| \delta \| \| \delta_x \| \\
		& \leq \left( \lambda_1 \| y_2 \|_{\infty} + \frac{\lambda_2 \| y_1 \|_{\infty}}{2} \right) \| \delta \|^2 + \left( \frac{\| y_2 \|_{\infty}}{\lambda_1} + \frac{\| y_1 \|_{\infty}}{2 \lambda_2} - \epsilon \right) \| \delta_x \|^2
\end{align*}
for any positive $\lambda_1$ and $\lambda_2$. We choose $\lambda_1 = 4 \| y_2 \|_{\infty} / \epsilon$ and $\lambda_2 = 2 \| y_1 \|_{\infty} / \epsilon$, with the result that
\begin{align*}
	\frac{d}{dt} \| \delta \|^2 
		& \leq \left( \frac{8 \| y_2 \|^2_{\infty} }{\epsilon} + \frac{2 \| y_1 \|^2_{\infty}}{\epsilon} \right) \| \delta \|^2 - \epsilon \| \delta_x \|^2 \\
		& \leq \left( \frac{8 \| y_2 \|^2_{\infty} }{\epsilon} + \frac{2 \| y_1 \|^2_{\infty}}{\epsilon} - \epsilon \pi^2 \right) \| \delta \|^2.
\end{align*}
We then apply Lemma \ref{lemma:gronwall} with $b(t) = 0$ to obtain the desired result.
\end{proof}


We now come to the main estimate of this section, which will yield the contractivity of $S_T$ for sufficiently large values of $T$.

\begin{prop} There exists a constant $A$, depending only on $\epsilon$, $\| H \|$, $\| u_1 \|$ and $\| u_2 \|$, such that
\begin{align*}
	\| \rho(0) \| \leq A  \left(1 + T \right) \left(1 + Z(T)^2 \right) \| u_1 - u_2 \| e^{-\epsilon \pi^2 T} 
\end{align*}
for any $T \geq 0$.
\label{prop:rho}
\end{prop}

\begin{proof}
Differentiating as in the previous sections (\textit{c.f.} Proposition 3.2 of \cite{W93}), we obtain\begin{align*}
	-\frac{d}{dt} \|\rho\|^2 \leq & \left( \frac{\| H \|^2}{\lambda_1} + \frac{\| y_{1} \|_{\infty}}{\lambda_2} + \frac{2 \| p_2 \|}{\lambda_3} - 2 \epsilon \right) \| \rho_x\|^2 \\
	&+ \lambda_2 \| y_{1} \|_{\infty} \| \rho \|^2 + \left( \lambda_1 \|H\|^2 + 2 \lambda_3 \|p_2 \| \right) \| \delta_x \|^2.
\end{align*}
for any positive $\lambda_1$, $\lambda_2$ and $\lambda_3$. We choose $\lambda_1 = 3 \|H\|^2 / \epsilon$, $\lambda_2 = 3 \| y_{1} \|_{\infty} / \epsilon$ and $\lambda_3 = 6 \|p_2\| / \epsilon$, and hence find that
\begin{align*}
	-\frac{d}{dt} \|\rho\|^2 &\leq \left( \frac{3 \| y_{1} \|_{\infty}^2}{\epsilon} - \epsilon \pi^2 \right) \| \rho\|^2 + \left( \frac{3 \|H\|^4}{\epsilon} + 6 \frac{\|p_2 \|^2}{\epsilon} \right) \| \delta_x \|^2.
\end{align*}
Since $\rho(T) = 0$, it follows from Lemma \ref{lemma:gronwall} and Proposition \ref{prop:Pbound} that
\begin{align}
	\| \rho(0) \|^2 \leq C e^{-\epsilon \pi^2 T} \int_0^T (1 + Z(T)^2) e^{\epsilon \pi^2 t} \| \delta_x(t) \|^2 dt.
\label{eqn:rho}
\end{align}

To bound the integral term, we first recall the inequality
\begin{align*}
	\epsilon \| \delta_x \|^2 \leq \left( \frac{8 \| y_2 \|^2_{\infty} }{\epsilon} + \frac{2 \| y_1 \|^2_{\infty}}{\epsilon} \right) \| \delta \|^2 - \frac{d}{dt} \| \delta \|^2
\end{align*}
found in the proof of Proposition \ref{prop:deltabound}. Then, up to a constant,
\begin{align*}
	\| \delta_x \|^2 \leq \left( \| y_{2x} \|^2 + \| y_{1x} \|^2 \right) \| u_1 - u_2 \|^2 e^{-\epsilon \pi^2 t} - \frac{d}{dt} \| \delta \|^2.
\end{align*}
Upon inserting this expression into Equation (\ref{eqn:rho}) we have
\begin{align*}
	\int_0^T (1 + Z(T)^2) e^{\epsilon \pi^2 t} \| \delta_x(t) \|^2 dt \leq & C (1 + Z(T)^2) \| u_1 - u_2 \|^2 \\
	&  - \left. \left( 1 + Z(t)^2 \right) e^{\epsilon \pi^2 t} \| \delta(t) \|^2 \right|_0^T \\
	& + \int_0^T \frac{d}{dt} \left[ \left( 1 + Z(t)^2 \right) e^{\epsilon \pi^2 t} \right] \| \delta(t) \|^2 dt,
\end{align*}
where we have integrated by parts and used the fact that $Z(t)$ is nondecreasing. The boundary term arising from the integration by parts is bounded above by $\| \delta(0) \|^2 = \| u_1 - u_2 \|^2$. For the last term we have
\begin{align*}
	\frac{d}{dt} \left[ \left( 1 + Z(t)^2 \right) e^{\epsilon \pi^2 t} \right] \| \delta(t) \|^2 \leq C \left(1 + \| z(t) \|_Z^2 + Z(t)^2 \right) \| u_1 - u_2 \|^2
\end{align*}
and the result follows.
\end{proof}


\section{Implications for data assimilation}
In this section we relate the previous analysis to the widely-used $4D$-Var data assimilation scheme. In $4D$-Var one is interested in minimizing a cost functional of the form
\begin{align*}
	J_{4D}(u) := \sum_{i=1}^N \left\| H \left. y(u) \right|_{t = t_i} - z_i \right\|_Z^2 + \beta \| u - \bar{u} \|_V^2,
\end{align*}
which is defined in terms of a finite set of observations $z_1, \ldots, z_N$, taken at times $t_1, \ldots, t_N$. The observation space $Z$ is usually taken to be $\mathbb{R}^m$ with a weighted Euclidean norm:
\begin{align*}
	\| z \|_Z^2 = \| R^{-1/2} z \|^2_{\textrm{Euc.}}
\end{align*}
where $R$ is the observational error covariance matrix.

The minimization problem is typically solved using a gradient descent (or conjugate gradient) method, where the gradient can be evaluated using Equation (\ref{eqn:gradJ}), or a discrete analog for the case of finitely many observations (see \cite{AAR10} for details). Thus each evaluation of $DJ$ requires integration of the forward model (\ref{eqn:Burgers}), followed by integration of the adjoint model (\ref{eqn:adjoint}) backwards in time.

This approach, while guaranteed to produce a minimizer for $J_T$, will not necessarily find the global minimum, even under the assumption that it is unique. Only in the event that $J$ has a unique \textit{local} minimum can we ensure the convergence of a steepest descent algorithm to this desired minimum. It is thus of great interest to know when this uniqueness condition is satisfied.

For the continuous-time functional, $J_T$, it was shown in \cite{W93} that this is true for sufficiently small observation times $T$, but also observed that this property likely fails for larger $T$, due to the nonlinearity of the forward model $u \mapsto y(u)$, and the resulting nonconvexity of $J_T$. We have shown in this paper that uniqueness also holds for sufficiently large $T$, given some additional, physically reasonable, assumptions on the set of potential minimizers. This arises as a result of the large-time exponential decay exhibited by solutions to Burgers' equation, which essentially demonstrates the dominance of the linear diffusive term over the nonlinear advective term, in the large-$T$ limit. These results, however, do not immediately apply to the uniqueness problem for the $4D$-Var cost functional, as they assume continuous-times observations---a technical convenience that is of course impossible to realize in the real world.

Thus an important next step is to modify the present analysis so that it is applicable to the case of discrete observations, i.e. the true $4D$-Var minimization problem. It will also be of great interest to apply these techniques to problems in higher spatial dimensions, and with more complicated nonlinearities.

While the present study has only considered the inverse problem for Burgers' equation from the variational perspective, there is in fact a great deal more information encoded in the cost function $J$ than just the location and uniqueness of extrema. There is an equivalent Bayesian formulation of the inverse problem in which $J$ corresponds to the log-likelihood of the posterior distribution $\mathbb{P}(u | z)$ (see \cite{S10} for details). In this interpretation, minimizers of $J$ correspond to modes of the posterior distribution, but of course there is much that one can say about a probability distribution beyond its modal structure. For applications, covariance information is of great importance, as it gives an indication of how trustworthy our solution to the inverse problem is. The nonlinearity of Burgers' equation means that higher moments of the posterior distribution will also contain nontrivial information, as would any other measure of the global shape of the distribution. It is in this formulation that we hope to better understand the limiting behavior of solutions to the inverse problem for dynamical systems with more severe nonlinearities and more complicated asymptotic structure.


\end{document}